\documentclass[11pt]{article}

\usepackage{amsmath,amssymb,amsthm,setspace,graphicx}
\usepackage[text={6.5in,8.4in}, top=1.3in]{geometry}
\usepackage{hyperref}

\newtheorem{theorem}{Theorem}[section]

\newtheorem{lemma}[theorem]{Lemma}

\newtheorem{problem}[theorem]{Problem}
\newtheorem{definition}[theorem]{Definition}

\renewcommand{\L}{{\cal L}}

\begin{document}

\title{Hamiltonian cycles above expectation in $r$-graphs and quasi-random $r$-graphs \thanks{This research was supported by the Israel Science Foundation (grant No. 1082/16).}}

\author{
Raphael Yuster
\thanks{Department of Mathematics, University of Haifa, Haifa
31905, Israel. Email: raphy@math.haifa.ac.il}
}

\date{}

\maketitle

\setcounter{page}{1}

\begin{abstract}
Let $H_r(n,p)$ denote the maximum number of Hamiltonian cycles in an $n$-vertex $r$-graph with density $p \in (0,1)$.
The expected number of Hamiltonian cycles in the random $r$-graph model $G_r(n,p)$
is $E(n,p)=p^n(n-1)!/2$ and in the random graph model $G_r(n,m)$ with $m=p\binom{n}{r}$ 
it is, in fact, slightly smaller than $E(n,p)$.

For graphs, $H_2(n,p)$ is proved to be only larger than $E(n,p)$ by a polynomial factor
and it is an open problem whether a quasi-random graph with density $p$ can be larger than $E(n,p)$
by a polynomial factor.

For hypergraphs (i.e. $r \ge 3$) the situation is drastically different. For all $r \ge 3$ it is proved that  $H_r(n,p)$ is
larger than $E(n,p)$ by an {\em exponential} factor and, moreover, there are quasi-random $r$-graphs with density $p$ whose number of Hamiltonian cycles is larger than $E(n,p)$ by an exponential factor.

\vspace*{3mm}

\end{abstract}

\section{Introduction}

All graphs and hypergraphs in this paper are finite and simple.
Let $r \ge 2$ be an integer. An $r$-graph (also called $r$-uniform hypergraph) $G$ has vertex set
$V(G)$ and edge set $E(G)$ which is a set of $r$-element subsets of $V$. The complete $r$-graph on $n$ vertices, denoted $K_n^r$, has, therefore, $\binom{n}{r}$ edges.
An $r$-graph $G$ on $n$ vertices has density $p$ if $|E(G)|=p\binom{n}{r}$.
In this paper we are interested in counting the number of Hamiltonian cycles
in $r$-graphs with a given constant density $p \in (0,1)$.
The problem of counting the number of Hamiltonian cycles and paths in graphs and hypergraphs has been extensively studied in various contexts \cite{AAR-2001,alon-1990,alon-2016,cuckler-2007,CK-2009,ERR-1971,FKS-2017,FK-2005,FMV-2019,frieze-2000,GLS-2017,janson-1994,KLO-preprint,KO-2012,KO-2014,szele-1943,thomassen-1985,wormald-p,yuster-2017}.

We begin with a definition of a Hamiltonian cycle. Suppose $G$ is an $r$-graph.
A Hamiltonian cycle of $G$ is a permutation of the vertices of $G$, say $v_1,\ldots,v_n$
such that for all $i=1,\ldots,n$ the set $\{v_i,\ldots,v_{i+r-1}\}$ is an edge (indices are modulo $n$).
We identify the Hamiltonian cycle with its set of $n$ edges, so observe that for $n \ge r+2$, each Hamiltonian cycle
is associated with precisely $2n$ permutations as rotations and reversing of the permutation yield the same cycle. In particular, $K_n^r$ has precisely $(n-1)!/2$ Hamiltonian cycles
for $n \ge r+2$ (if $r=2$ then this holds also for $n=3$).
Given a permutation of $n$ vertices, its {\em $r$-set} is the set of $n$ edges of its corresponding
Hamiltonian cycle in $K_n^r$.
We note also that there are some looser notions of Hamiltonicity
where the intersection of two consecutive edges of the cycle is allowed to be of order less than $r-1$ but we are not concerned with these looser notions here.

As our aim is to count Hamiltonian cycles in $r$-graphs with a given density, let $H(G)$ denote
the number of Hamiltonian cycles in $G$. Clearly, for any\footnote{One can assume that $p$ is rational in order to have infinitely many $n$ for which there are $r$-graphs with $n$ vertices and density $p$.} given density $p \in (0,1)$
there are $r$-graphs with density $p$ that are non-Hamiltonian and on the other hand, it is clear that every $r$-graph with density $p$ cannot have too many Hamiltonian cycles.
Thus, the extremal parameter of interest is $H_r(n,p)$, the maximum possible value of $H(G)$ ranging over all $n$-vertex $r$-graphs with density $p$. Note that we assume that $n$ is such that $p\binom{n}{r}$ is an integer
as otherwise $H_r(n,p)$ is undefined.

To estimate $H_r(n,p)$, it is natural to consider the number of Hamiltonian cycles expected in random $r$-graphs as this tells us what (approximately) $H(G)$ typically is.
This approach is also taken in most other papers that consider counting the number of Hamiltonian cycles.
The two standard models to consider are the Erd\H{o}s-R\'enyi random $r$-graph model $G_r(n,p)$ where
each edge is uniformly and independently selected for $G \sim G_r(n,p)$ with probability $p$, 
and the Erd\H{o}os-R\'enyi random $r$-graph model $G_r(n,m)$ where precisely $m$ edges of the
$\binom{n}{r}$ possible ones are selected for $G \sim G_r(n,m)$.
It is immediate that the expected value of $H(G)$ for $G \sim G_r(n,p)$ is $p^n(n-1)!/2$
as each Hamiltonian cycle of $K_n^r$ remains such in $G$ with probability $p^n$,
so we call this quantity the {\em expectation value} and denote $E(n,p)=p^n(n-1)!/2$.
Unfortunately, we cannot use  this argument to say that $E(n,p)$ is a lower bound for $H_r(n,p)$ since $G \sim G_r(n,p)$ only
has an {\em expected} density of $p$, so it could have larger density, and the number of edges is positively correlated with $H(G)$. On the other hand, for $G \sim G_r(n,p\binom{n}{r})$
the expected value of $H(G)$ is less than $E(n,p)$ by a constant factor. To see this, just observe that
for any given nontrivial prefix of a Hamiltonian cycle (namely, a path with at least one edge), if one tries to extend it by choosing the next vertex uniformly at random from the vertices that are not yet on the path, the probability of the next edge to exist is slightly smaller than $p$.
On the positive side, the expectation of $H(G)$ for $G \sim G_r(n,p\binom{n}{r})$ is obviously a lower bound for $H_r(n,p)$, so we have that $H_r(n,p) \ge c_p E(n,p)$ where $c_p < 1$ is a constant depending on $p$.

So, what can we say about $H_r(n,p)$? Is it always larger than $E(n,p)$ and if so, by how much?
The answer, though always positive, is very different with respect to ``how much'' depending on whether $r=2$ or $r > 2$, and in a strong sense which we now make precise.
For a parameter $\epsilon > 0$, we say that an $r$-graph with $n$ vertices is
{\em $(\epsilon,p)$}-quasi random if its density is $p$ and for every subset $W \subset V(G)$ with $\lfloor n/2 \rfloor$
vertices, the density of $G[W]$ is $p\pm\epsilon$ (namely, the density of the $r$-graph induced by $W$
is in $(p-\epsilon,p+\epsilon)$). For $r=2$, this definition of quasi-randomness is equivalent to
several other definitions, as first proved by Chung, Graham and Wilson \cite{CGW-1989}. For $r \ge 3$ there are other notions of
quasi-randomness (see \cite{FR-2002,gowers-2006}) and the one we defined here is usually called weak quasi-randomness although our
 results stated later could easily be shown to hold for other notions of hypergraph 
 quasi-randomness. Clearly, for every $\epsilon > 0$, an element chosen from $G_r(n,p\binom{n}{r})$ is 
 $(\epsilon,p)$-quasi random asymptotically almost surely. Given that the expectation of $H(G)$ chosen 
from this model
is always smaller than $E(n,p)$, the following open problem seems of natural interest.
Define $H_r(n,p,\epsilon)$ as the maximum of $H(G)$ ranging over all $n$-vertex $r$-graphs which are
$(\epsilon,p)$-quasi random. Observe that trivially $H_r(n,p,1)=H_r(n,p)$.
\begin{problem}\label{prob:1}
	Let $\epsilon > 0$ and $p \in (0,1)$ be fixed. Determine or estimate 
	\begin{equation}\label{e:0}
	\frac{H_r(n,p,\epsilon)}{E(n,p)}\;.
	\end{equation}
\end{problem}
So, by the arguments above, we know that (\ref{e:0}) is bounded from below by a constant. As can be seen from the following theorems, in the case $r=2$, (\ref{e:0}) is {\em upper-bounded} by
a polynomial in $n$ while for $r \ge 3$, (\ref{e:0}) is {\em lower-bounded} by an exponential in $n$.
In fact, for the case $r=2$ we do not even know how to lower-bound (\ref{e:0}) by a polynomial in $n$ for every $\epsilon > 0$ (we do know that for, say, $\epsilon=1$ though) so we raise the following specific problem.
\begin{problem}\label{prob:2}
	Let $\epsilon > 0$ and $p \in (0,1)$ be fixed. Determine whether the following holds:
	$$
	\lim_{n \rightarrow \infty} \frac{H_2(n,p,\epsilon)}{E(n,p)} = \infty\;.
	$$
\end{problem}

The following theorem proves that $H_2(n,p)$ is polynomially larger than $E(n,p)$ but not more than that.
\begin{theorem}\label{t:graphs}
	Let $p \in (0,1)$. There exists a constant $c_p$ depending on $p$ such that
	$$
	(1+o(1))c_p n^{\frac{1}{2}}E(n,p) \le H_2(n,p) \le (1+o(1))\frac{1}{e}(\sqrt{2\pi})^{\frac{1}{p}-1}p^{\frac{1}{2p}}n^{\frac{1}{2}+\frac{1}{2p}}E(n,p)\;.
	$$
\end{theorem}
	While the proof of the upper bound follows rather immediately from a result of Alon \cite{alon-1990}, the proof of the lower bound is somewhat more delicate as we need to construct graphs with density {\em precisely} $p$ having many Hamiltonian cycles. For some densities there are relatively simple explicit constructions, while for other densities, the construction is probabilistic.

We now turn to hypergraphs where the proofs become considerably more involved.
Our main result is the following.
\begin{theorem}\label{t:hypergraphs}
	Let $\epsilon > 0$ and $p \in (0,1)$ be fixed. Then for all $r \ge 3$,
	$$
	\frac{H_r(n,p,\epsilon)}{E(n,p)} \ge 2^{\Omega(n^{\gamma})}
	$$
	where $\gamma=\frac{1}{2}-o_n(1)$ if $r \ge 4$ and $\gamma=\frac{1}{3}-o_n(1)$ if $r=3$.
\end{theorem}
So not only is this in sharp contrast with the $r=2$ case, but even if we require the graphs to be quasi-random we can still have more Hamiltonian cycles than the expectation value by an exponential factor. 
If we do not require quasi-randomness, the exponent can be made linear in $n$, as the following theorem asserts.
\begin{theorem}\label{t:hypergraphs-2}
	Let $p \in (0,1)$ be fixed. Then for all $r \ge 3$,
	$$
	\frac{H_r(n,p)}{E(n,p)} \ge \left(\frac{(k-r+1)!k^{r-1}}{k!}\right)^{n-o(n)}
	$$
	where $k \ge r+1$ is the least integer satisfying $\frac{k!}{(k-r)!k^r} > p$.
\end{theorem}
In the forthcoming sections we first prove the graph theoretic theorem \ref{t:graphs} in Section 2.
In Section 3, we will prove a version of Theorem \ref{t:hypergraphs} that applies only to $r=3$
and to infinitely many $n$, but not all $n$. That version has an important feature of being highly symmetric as the constructed quasi-random $3$-graph uses a certain spherical-geometric combinatorial design (which is known to exist only for $r=3$). Some of the ideas used in the proof of Theorem \ref{t:hypergraphs} become simpler in this case,
so Section 3 also serves as a gentle introduction to the proof of Theorem \ref{t:hypergraphs} appearing in Section 4. The proof of Theorem \ref{t:hypergraphs-2} is in Section 5.

\section{Graphs and \texorpdfstring{$H_2(n,p)$}{H2(n,p)}}

In this section we prove Theorem \ref{t:graphs}.

\subsection{Upper bound}
As mentioned in the introduction, the upper bound of Theorem \ref{t:graphs} follows almost directly from a result of Alon \cite{alon-1990} which we state after introducing the following notation.
For a graph $G$, let $A_G$ denote its adjacency matrix. Recall that $A_G$ is a symmetric binary
matrix whose rows and columns are indexed by $V(G)$ and $A_G(i,j)=1$ if and only if $ij \in E(G)$.
A {\em generalized $2$-factor} of $G$ is a spanning subgraph whose components are simple cycles
and single edges. Let $F(G)$ denote the number of generalized $2$-factors of $G$,
let $F_k(G)$ denote the number of generalized $2$-factors of $G$ with exactly $k$ cycles,
and notice that $F_1(G)=H(G)$. Recall that the {\em permanent} of a square matrix $A$, denoted by $Per(A)$  is its unsigned determinant. It is immediate to verify that
\begin{equation}\label{e:1}
Per(A_G) = \sum_{k=0}^{\infty} 2^k F_k(G)\;.
\end{equation}
By Br\'egman's Theorem \cite{bregman-1973}, for a binary $n \times n$ matrix $A$ it holds
that $Per(A) \le \Pi_{i=1}^n (r_i!)^{1/r_i}$ where $r_i$ is the number of ones in row $i$.
Now, as proved by Alon, subject to $A$ having $2p\binom{n}{2}$ ones, 
$Per(A)$ is maximized if the $r_i$ are as equal as possible; namely, if $p(n-1)$ is an integer
then all $r_i$ equal
$p(n-1)$ and otherwise some of them equal $\lfloor p(n-1) \rfloor$ and some of them
equal $\lceil p(n-1) \rceil$.
Plugging this into Br\'egman's bound, Alon deduced that \footnote{In fact, Alon proved this for the case $p=\frac{1}{2}$ but the exact same argument holds for every $p \in (0,1)$.}
$$
Per(A) \le (1+o(1))\frac{1}{e}(\sqrt{2\pi})^{\frac{1}{p}-1}n^{\frac{1}{2}+\frac{1}{2p}}p^{n+\frac{1}{2p}}(n-1)!\;.
$$
Now, if $G$ is a graph with $n$ vertices and density $p$ then $A_G$ indeed has precisely
$2p\binom{n}{2}$ nonzero entries. It therefore follows from (\ref{e:1}) and the last inequality that
\begin{align*}
H(G) = F_1(G) & \le \frac{1}{2}Per(A_G) \\
& \le (1+o(1))\frac{1}{e}(\sqrt{2\pi})^{\frac{1}{p}-1}p^{\frac{1}{2p}}
n^{\frac{1}{2}+\frac{1}{2p}}p^{n}\frac{(n-1)!}{2}\\
& = (1+o(1))c_p n^{\frac{1}{2}+\frac{1}{2p}}E(n,p)
\end{align*}
where $c_p = \frac{1}{e}(\sqrt{2\pi})^{\frac{1}{p}-1}p^{\frac{1}{2p}}$.

\subsection{Lower bound}

It will be convenient to use the following lemma that says that once we know the number of Hamiltonian cycles in a certain $r$-graph with density $q$, we can use that $r$-graph to obtain a lower bound for $H_r(n,p)$ for $p \le q$.
\begin{lemma}\label{l:q-p}
Let $0 \le p \le 1$ be given. Suppose  that $G$ is an $n$-vertex $r$-graph with density $q \ge p$.
Then $H_r(n,p) \ge (p/q)^n e^{-2/p} H(G)(1-o(1))$.
\end{lemma}
\begin{proof}
Let $G$ be an $r$-graph with $n$ vertices and density $q$ where $q \ge p$.
We construct graphs with density precisely $p$ as follows.
Define a symmetric probability space ${\cal P}(G,p)$ over all the spanning subgraphs of $G$ with
precisely $p\binom{n}{r}$ edges\footnote{Recall, we assume $p\binom{n}{r}$ is an integer.}. Observe that since $G$ has at least $q\binom{n}{r} \ge p\binom{n}{r}$ edges,
${\cal P}(G,p)$ is well-defined.

We compute the expected number of Hamiltonian cycles in ${\cal P}(G,p)$ which we denote by
$\mathbb{E}[G,p]$ and note that this lower-bounds $H_r(n,p)$.
Fix some Hamiltonian cycle $C$ of $G$ and let $G^* \sim {\cal P}(G,p)$.
We compute the probability that $C$ corresponds to a Hamiltonian cycle of $G^*$, so
let $X_C$ be the corresponding indicator random variable for this event.
We expose the $n$ edges of $C$ one by one, to see if they are present in $G^*$.
The probability that the first exposed edge is an edge of $G^*$ is $p/q$ since this is the fraction of edges of $G$ taken to $G^*$.
The probability that $i$th exposed edge is an edge of $G^*$ {\em given} that the previous ones were edges
of $G^*$ is 
\begin{align*}
\frac{p\binom{n}{r}-i+1}{q\binom{n}{r}-i+1} & \ge \frac{p\binom{n}{r}-n+1}{q\binom{n}{r}-n+1}\\
& \ge \frac{pn/2-1}{qn/2-1}\\
& = \frac{p}{q}\left(1-\frac{q/p-1}{qn/2-1}\right)\\
& > \frac{p}{q}\left(1-\frac{2/p}{n}\right)\;.
\end{align*} 
It follows that
$$
\Pr[X_C=1] \ge (p/q)^n\left(1-\frac{2/p}{n}\right)^n = (p/q)^n e^{-2/p}(1-o(1))\;.
$$
Hence we obtain that
$$
\mathbb{E}[G,p] \ge H(G)(p/q)^n e^{-2/p}(1-o(1))\;.
$$
\end{proof}

We first prove the lower bound for the case $p=\frac{1}{2}$. In this case the construction
is explicit. First notice that any good construction should not have small cuts, as every Hamiltonian cycle must go through an edge of the cut and that severely limits the number of possible
Hamiltonian cycles. Going to the other extreme, it seems as a good idea to try a complete balanced bipartite graph but this has density slightly above $1/2$.
Nevertheless, since when $n$ is even a complete balanced bipartite graph with $n/2$ vertices in each side has
$(n/2)!(n/2-1)!/2$ Hamiltonian cycles which is slightly above the expectation value, it is profitable
to consider slight variations of it.
It turns out that by changing a small amount of edges and non-edges, we can obtain an almost complete balanced almost bipartite graph with density at most $1/2$ and which already has many Hamiltonian cycles, in fact an amount that is a polynomial factor larger than $H_2(n,\frac{1}{2})$. 

Suppose first that $n$ is even.
Construct the graph $B_n$ by taking a $K_{n/2,n/2}$ and removing a perfect matching
($B_n$ is also known as the {\em crown graph}).
Observe that the density of this graph is actually slightly smaller than $1/2$.
Suppose that the sides of $B_n$ are $X$ and $Y$ and that the vertices of $X$ are the odd integers in
$[n]$ and the vertices of $Y$ are the even integers in $[n]$. Then we can assume that the edges \footnote{We denote edges connecting $x$ and $y$ by $xy$ unless this is confusing, in which case the notation $(x,y)$ is used.}
$(2i-1,2i)$ are missing from $B_n$ for $i=1,\ldots,n/2$ and all other edges of the form $ij$ where
$i \in X$ and $j \in Y$ are in $B_n$.
The number of Hamiltonian cycles in $B_n$ has a well-known closed formula expressed by summation (see \cite{alekseyev-2016}). Here we lower-bound this amount with a summation-free expression
and in terms of $E(n,{\textstyle \frac{1}{2}})$.

Consider some permutation $\pi \in S_n$. We call $\pi$ {\em good} if (i) all the vertices of $X$ are in the odd locations of $\pi$ and (ii) for $i=1,\ldots,n/2$ it holds that location $\pi^{-1}(2i-1)$ is not adjacent
to location $\pi^{-1}(2i)$ (we think of adjacency as cyclic so locations $n$ and $1$ are considered adjacent).
For example, if $n=8$ then the permutation $14527638$ is good since $1,3,5,7$ are in the odd locations,
location $\pi^{-1}(1)=1$ is not adjacent to location $\pi^{-1}(2)=4$, similarly $\pi^{-1}(3)=7$ is not adjacent to $\pi^{-1}(4)=2$, $\pi^{-1}(5)=3$ is not adjacent to $\pi^{-1}(6)=6$, $\pi^{-1}(7)=5$ is not adjacent to $\pi^{-1}(8)=8$.
Clearly, by the definition of $B_n$, every good $\pi$ corresponds to a Hamiltonian cycle of $B_n$.
Let $P(n)$ denote the number of good permutations. Hence, $H(G) \ge P(n)/n$ as the automorphism group of an undirected  Hamiltonian cycle is of order $2n$ but in good permutations the first vertex is  always from $X$.

We now turn to computing $P(n)$. Suppose we have placed all $n/2$ elements of $Y$ in the even locations of
the permutation. There are $(n/2)!$ ways to do that. It now remains to place the elements of $X$ in the
remaining (odd) locations, but subject to the location adjacency restrictions. As we would not like, say, $1$ to be next to $2$, this designates two locations in which $1$ is not allowed to be placed. Similarly, each of $1,3,5,\ldots,n-1$ is forbidden two locations. Define a bipartite graph $Q$ on sides $X$ and $Z$
where $Z$ are all the $n/2$ locations to place the elements of $X$ and there is an edge between $x \in X$
and $z \in Z$ if we are not allowed to place $x$ in location $z$. So, the vertices $1,3,\ldots,n-1$
of $X$ have degree $2$ in $Q$ and
each vertex of $Z$ also has degree $2$ since a location is not allowed for precisely two elements of $X$. Hence, $Q$ is just a $2$ factor (in fact, $Q$ is a Hamiltonian cycle).

It is now a standard inclusion-exclusion argument to determine the number of ways we can place
the elements of $X$ subject to the restrictions defined by $Q$. An alternative convenient way to count this
is as follows. Define an $n/2 \times n/2$ matrix $A_{Q}$ where the rows are indexed by $X$ and the columns by $Z$, and $A_{Q}(i,j)=1/(n/2-2)$ if $ij$ is not an edge of $Q$ (so in particular,
$i$ is allowed to be placed in location $j$). Otherwise, $A_{Q}(i,j)=0$.
As each row and column of $A_{Q}$ has only two zero entries, $A_{Q}$ is a doubly stochastic matrix.
By the well-known Theorems of Erorychev \cite{erorychev-1981} and Falikman \cite{falikman-1981}, $Per(A) \ge (n/2)!/(n/2)^{n/2}$. As each nonzero
permutation of $A_{Q}$ equals $(1/(n/2-2))^{n/2}$ and each such permutation corresponds to a placement of the elements of $X$ in allowed locations, we have that
\begin{align}
P(n) & \ge (n/2)! \cdot \frac{(n/2)!}{(n/2)^{n/2}} \cdot (n/2-2)^{n/2} \nonumber \\
&  \ge \frac{1}{e^2}\left((n/2)!\right)^2(1-o(1)) \nonumber\\
& = \frac{\sqrt{2\pi}}{e^2}n^{3/2} \frac{(n-1)!}{2^{n+1}}(1-o(1)) \label{e:2}\;.
\end{align}
It follows that
\begin{equation}\label{e:3}
H_2(n,{\textstyle \frac{1}{2}}) \ge \frac{\sqrt{2\pi}}{e^2}n^{1/2} \frac{(n-1)!}{2^{n+1}}(1-o(1)) =
\frac{\sqrt{2\pi}}{e^2}n^{1/2} E(n,{\textstyle \frac{1}{2}})(1-o(1))\;.
\end{equation}

Next we consider the case where $n$ is odd. In this case we define $B_n$ as follows.
Take $B_{n-1}$ (which is already defined since $n-1$ is even) and add
a new vertex named $n$. Connect vertex $n$ to precisely $\lfloor 3(n-1)/4 \rfloor$ other vertices (it does not matter which).
The number of edges of $B_n$ is $\lfloor n(n-1)/4 \rfloor$ so its density is at most $1/2$.
Now, consider some good permutation of $S_{n-1}$ and recall there are $P(n-1)$ such permutations.
We can extend this permutation to a permutation of $S_n$ in $n$ possible ways by placing vertex $n$
somewhere. However, we would not like to place $n$ in locations that are adjacent to its non-neighbors.
Since vertex $n$ has only $\lceil (n-1)/4 \rceil$ non-neighbors, there are at most $(n+1)/2$ non-allowed locations, which leaves at least $(n-1)/2$ possible locations to place vertex $n$. Now observe that each such permutation
corresponds to a Hamiltonian cycle of $B_n$. Hence, if $P(n)$ denotes the number of good permutations in this case, we have $P(n) \ge P(n-1) \cdot (n-1)/2$, so by (\ref{e:2})
\begin{align*}
P(n) & \ge \frac{n-1}{2} P(n-1)\\
&   \ge \frac{n-1}{2} \frac{\sqrt{2\pi}}{e^2}n^{3/2} \frac{(n-2)!}{2^{n}}(1-o(1))\\
& = \frac{\sqrt{2\pi}}{e^2}n^{3/2} \frac{(n-1)!}{2^{n+1}}(1-o(1))\;.
\end{align*}
Hence, (\ref{e:3}) holds in this case as well.

We next consider densities other than $1/2$. Assume first that $p < \frac{1}{2}$ is fixed and that $n$ is such that $p\binom{n}{2}$ is an integer. 
We will use Lemma \ref{l:q-p} with $G=B_n$ as above, and observe that $q\le 1/2$ and that for
$n$ sufficiently large, $p < q$. As we have just proved that the right hand side of (\ref{e:3}) lower bounds $H(B_n)$, the lemma implies that
$$
H_2(n,p) \ge (2p)^n e^{-2/p} \frac{\sqrt{2\pi}}{e^2}n^{1/2} E(n,{\textstyle \frac{1}{2}})(1-o(1))
= \frac{\sqrt{2\pi}}{e^{2+2/p}}n^{1/2}E(n,p)(1-o(1))\;.
$$

Assume next that $p > 1/2$. Recall the Tur\'an Graph $T(n,k)$ which is the complete balanced $k$-partite
graph where every part is of size $\lceil n/k \rceil $ or $\lfloor n/k \rfloor$.
Observe that its density is only slightly larger than $(k-1)/k$. In fact, its density is smaller than
$(1+2/n)(k-1)/k$. The number of Hamiltonian cycles in $T(n,k)$ corresponds to the set of cyclic permutations of the vertices of $T(n,k)$ where no two adjacent vertices from the same part are next to each other (again, each Hamiltonian cycle corresponds to $2n$ such permutations). This latter problem is an old problem posed by Smirnov that was solved asymptotically in \cite{smirnov-1966} (in fact, it was solved for all complete multipartite graphs, not just balance ones). In particular, it follows from \cite{smirnov-1966} that
\begin{equation}\label{e:tkn}
H(T(n,k)) = (1-o(1))c_k n^{1/2} E(n,{\textstyle \frac{k-1}{k}})
\end{equation}
where $c_k$ is a small constant depending only on $k$. 
Let, therefore, $k$ be the least integer such that $(k-1)/k > p$
and that $n$ is such that $p\binom{n}{2}$ is an integer.
We use Lemma \ref{l:q-p} with $G=T(n,k)$.
The lemma, together with (\ref{e:tkn}) implies that
\begin{align*}
H_2(n,p) & \ge \left(\frac{pk}{(k-1)(1+2/n)}\right)^n e^{-2/p} (1-o(1))c_k n^{1/2} E(n,{\textstyle \frac{k-1}{k}})\\
& = \left(\frac{1}{(1+2/n)}\right)^n e^{-2/p} (1-o(1))c_k n^{1/2} E(n,p)\\
& = (1+o(1))c_p n^{1/2} E(n,p)\\
\end{align*}
where $c_p$ depends only on $k$ and $p$, and therefore only on $p$.
\qed

\section{\texorpdfstring{$3$}{3}-graphs and \texorpdfstring{$H_3(n,p,\epsilon)$}{H3(n,p,epsilon)} using Steiner systems}

In this section we prove a slightly weaker version of Theorem \ref{t:hypergraphs} that applies only to the case $r=3$. This version has the property
that the constructed $3$-graph can be partitioned into isomorphic $p$-dense ``chunks'' in a highly symmetric way that we do not have for $r \ge 4$. Nevertheless, many of the ingredients appearing in the proof of
Theorem \ref{t:hypergraphs} are already present in the proof we present in this section.
For simplicity, we will assume in this section that $1/p$ additionally satisfies certain divisibility constraints but it is very easy to deduce the result for all $p \in (0,1)$ using Lemma \ref{l:q-p}, or as we show in the
proof of Theorem \ref{t:hypergraphs} in the following section. We will also assume that $n$ is of a certain form (and infinitely many integers have this form).
Again, it is not too difficult to extend the result here to all $n$ at the price of some technical modifications.
But as the main result in Section 4 already handles the general case (also for $r=3$), we prefer to keep things simple in this section. The exact statement of the theorem we prove here follows.
\begin{theorem}\label{t:special}
	Let $\epsilon > 0$ and let $p \in (0,1)$ be such that $1/p$ is an integer.
	Suppose that $n=q^4+1$ where $q$ is a prime power and $pq^3(q^6+q^4+q^2+1)$ is an integer (this occurs, say, if $q$ is a power of $2$ and $p=\frac{1}{2}$). Then:
	$$
	\frac{H_3(n,p,\epsilon)}{E(n,p)} \ge (1/p)^{\Theta(n^{1/4})}\;.
	$$	
\end{theorem}
\begin{proof}
Recall that a Steiner system $S(r,k,n)$ is a set $L$ of $k$-element subsets of $[n]$
such that every $r$-element subset of $[n]$ is contained in exactly one element of $L$.
The elements of $L$ are called {\em blocks}.
In hypergraph terminology, it is a {\em decomposition} of $K_n^r$ into pairwise edge-disjoint subgraphs isomorphic to $K_k^r$.
While for $r \ge 4$ no infinite families of Steiner systems are known for large $k$ that is {\em polynomial} in $n$,
for $r=3$ such a family can be constructed from spherical geometries. In particular, if $q$ is a prime power and
$s > 0$ is an integer then an $S(3,q+1,q^s+1)$ exists \cite{CD-2010,hughes-1965}.
For our construction here we will use $s=4$, namely $S(3,q+1,q^4+1)$.
Notice that the number of blocks in this case is $|L|=q^3(q^6+q^4+q^2+1) \approx q^9$.

Recall that in our theorem's statement we assume that $n=q^4+1$ where $q$ is a prime power and that our density $p$ is such that $pq^3(q^6+q^4+q^2+1)$ is an integer and that $1/p$ is an integer.
Consider an $S(3,q+1,n)$ and its set of blocks $L$ and notice that $|L| \approx n^{9/4}$ and that $p|L|$
is an integer.
We can further partition $L$ into $p|L|$ parts denoted by $L_1,\ldots,L_t$ where
$t=p|L|$ and each $L_i$ consists of precisely $1/p$ blocks of size $q+1$ each.
However, we will require our partition to satisfy a stronger property:
for each $L_i$, all the $1/p$ blocks in $L_i$ are pairwise disjoint.
It is not difficult to find such a partition. Consider a graph $M$ whose vertex set is $L$
and two blocks are connected if they have nonempty intersection. The number of vertices in this graph
is $|L| \approx n^{9/4}$. As each $v \in [n]$ is in precisely $\binom{n-1}{2}/\binom{q}{2}$ blocks,
the maximum degree of $M$ is at most $(q+1)\binom{n-1}{2}/\binom{q}{2} < n^2$.
Hence by the Hajnal-Szemer\'edi Theorem \cite{HS-1970}, if $n$ is sufficiently large (recall that $p$ is constant) $M$ has $|L|p$ pairwise disjoint independent sets of order $1/p$ each. 

We now construct our $3$-graph $G$ with edge density $p$ which will be used to prove a lower bound for
$H_3(n,p,\epsilon)$. The vertex set of $G$ is $[n]$, namely the same vertex set used above for
the Steiner system. From each $L_i$ choose one of its $1/p$ blocks at random (all $t$ choices made independently), and add to $G$ all possible edges that this block contains (namely, a copy of $K_{q+1}^3$).
The remaining $1/p-1$ blocks from each $L_i$ remain edgeless. Notice that the obtained $3$-graph has density precisely $p$, as required. Furthermore since each $3$-set is an edge with probability $p$ and the event that a given $3$-set
is an edge depends only on  $O(q^3/p)=o(n^3)$ other $3$-sets, it is straightforward that for every $\epsilon > 0$, if $q$ (and hence $n$) is sufficiently large then $G$ is $(\epsilon,p)$-quasi random with probability $1-o(1)$. We will prove that with high probability, $G$ as above has many Hamiltonian cycles, enough to yield Theorem \ref{t:special}.

Consider a permutation $\pi$ of $[n]$ together with its $3$-set $S(\pi)$ (recall that the
$3$-set of a permutation consists of the $n$ edges of the Hamiltonian cycle in $K_n^3$ corresponding to the permutation).
We call $\pi$ {\em good} if no $L_i$ contains two elements of $S(\pi)$ in distinct blocks of $L_i$. Otherwise it is {\em bad}.
Note that in a good permutation, $L_i$ is allowed to contain several elements of $S(\pi)$ as long as they are all in the same block. Clearly, for every bad permutation $\pi$, the corresponding $S(\pi)$ is not a Hamiltonian cycle of $G$ regardless of the coin tosses that generated $G$.
For a good permutation $\pi$ let
$$
f(\pi) = |\{i\,:\, \exists B \in L_i, \exists e \in S(\pi),  e \subset B\}|\;.
$$
So $f(\pi)$ is the number of parts $L_i$ that contain a block that covers some element of $S(\pi)$.
Observe that trivially $f(\pi) \le |S(\pi)|=n$.
\begin{lemma}\label{l:prob}
	Let $\pi$ be a good permutation. The probability that $S(\pi)$ is a Hamiltonian cycle of $G$
	is $p^{f(\pi)}$.
\end{lemma}
\begin{proof}
	For every $i$ such that $L_i$ contains at least one element of $S(\pi)$, the probability that all of the elements of $S(\pi)$ appearing in $L_i$ occur as edges in $G$ is $p$. As two $3$-sets from distinct parts are independent with respect to the event of being an edge of $G$, the lemma follows.
\end{proof}
We will prove that most permutations are good and, moreover, that for most good permutations, the corresponding $f(\pi)$ is significantly smaller than $n$. Together with Lemma \ref{l:prob} this will imply the existence of many Hamiltonian cycles in some $G$.

\begin{lemma}\label{l:num-bad}
	The number of bad permutations is at most $n! \cdot \Theta(n^{-1/4})$.
\end{lemma}
\begin{proof}
	It is equivalent, although more convenient, to upper-bound the probability that
	a random permutation $\pi \in S_n$ is bad. So, let $\pi \in S_n$ be drawn at random and recall that its
	$3$-set $S(\pi)$ contains $n$ elements.
	Let $e,f \in S(\pi)$ be two distinct elements. First observe that if $e \cap f \neq \emptyset$ then
	$e$ and $f$ never appear in two different blocks of the same $L_i$ since all blocks in $L_i$ are
	vertex-disjoint. So we can assume that $e \cap f = \emptyset$.
	Suppose that $e$ is in some block of some $L_i$. What is the probability that $f$ is in one of the other
	$1/p-1$ blocks of $L_i$? Let $B$ be such a block. As the $3$-set $f$ is a completely random triple of
	$[n] \setminus e$, the probability that $f \subset B$ is $(q+1)q(q-1)/(n-3)(n-4)(n-5)$.
	Hence, by the union bound, the probability that $e$ and $f$ are in two distinct blocks of the same part
	is at most
	$$
	({\textstyle \frac{1}{p}}-1)\frac{(q+1)q(q-1)}{(n-3)(n-4)(n-5)} = \Theta(n^{-9/4})\;.
	$$
	As there are less than $n^2$ pairs of distinct elements of $S(\pi)$ it follows again from the union bound and the last inequality that the probability that $\pi$ is bad is at most $\Theta(n^{-1/4})$.
\end{proof}

Let $\L$ denote the set of good permutations. Then by the last lemma, $|\L| \ge n! \cdot (1-\Theta(n^{-1/4}))$. Observe also that by Lemma \ref{l:prob} the random variable $H(G)$ which is the number of Hamiltonian cycles in $G$ satisfies
\begin{equation}\label{e:EHG}
\mathbb{E}[H(G)] = \frac{1}{2n}\sum_{\pi \in \L} p^{f(\pi)}\;.
\end{equation}
Let $\overline{f}$ denote the average of $f(\pi)$ ranging over all $\pi \in \L$. Then we have:
\begin{lemma}\label{l:amgm}
\begin{equation}\label{e:amgm}
\mathbb{E}[H(G)] \ge \frac{|\L|}{2n}p^{\overline{f}}\;.
\end{equation}
\end{lemma}
\begin{proof}
	The lemma follows from (\ref{e:EHG}) and the inequality of arithmetic and geometric means.
\end{proof}
So, our remaining task is to lower-bound (\ref{e:amgm}), i.e., the right-hand side of the inequality
in Lemma \ref{l:amgm}.
To this end, define for $\pi \in S_n$ the value $g(\pi)$ to be the number of {\em consecutive} pairs of elements of $S(\pi)$ that fall into the same block. Note that a consecutive pair corresponds to a consecutive $4$-tuple of $\pi$ (where $\pi$ is viewed cyclically).
Observe that if $\pi$ is good, then $f(\pi) \le n - g(\pi)$ as the
worst case is when all the elements of $S(\pi)$ in a given block are consecutive.
For example, suppose that $\pi=id$, that $g(\pi)=3$, that $(1,2,3),(2,3,4),(3,4,5)$ are in the same block $B$ and
that $(6,7,8),(7,8,9)$ are in the same block $B'$. Then, if $B \neq B'$ it could be that the remaining $n-5$ elements of $S(\pi)$ are in other distinct blocks so we can only promise that
$f(\pi) \le n-3$ while if $B=B'$ we know that $f(\pi) \le n-4$.
Let $\overline{g}$ denote the average of $g(\pi)$
ranging over all $\pi \in \L$ and let $\overline{g^*}$ denote the average of $g(\pi)$
ranging over all $n!$ permutations. Since $f(\pi) \le n - g(\pi)$ for $\pi \in \L$
we have that $\overline{f} \le n - \overline{g}$.
We will compute $\overline{g^*}$, then we will prove that
$\overline{g^*}$ is very close to $\overline{g}$ and then we will use $\overline{f} \le n - \overline{g}$ to prove an estimate for $\overline{f}$ which will complete our task using (\ref{e:amgm}).

\begin{lemma}\label{l:g^*}
$$
\overline{g^*} = n(q-2)/(n-3) = \Theta(n^{1/4})\;.
$$
\end{lemma}
\begin{proof}
Again, it is convenient to consider a permutation $\pi$ chosen at random from $S_n$.
Consider two consecutive elements of $S(\pi)$. Denote them by $(a,b,c)$ and $(b,c,d)$. The probability that
$(b,c,d)$ is in the same block as $(a,b,c)$ is precisely $(q-2)/(n-3)$ as $d$ has to be chosen to the remaining $q-2$ elements of the block containing $(a,b,c)$ out of the remaining $n-3$ vertices $[n] - \{a,b,c\}$. As there are $n$ distinct pairs of consecutive elements of $S(\pi)$, the expected value of
$g(\pi)$ for a randomly chosen permutation is $n(q-2)/(n-3)$. As this expected value is
precisely $\overline{g^*}$, the lemma follows.
\end{proof}

\begin{lemma}\label{l:g}
	$\overline{g} \ge \overline{g^*} - \Theta(n^{1/12}) $.
\end{lemma}
\begin{proof}
Consider some $\pi \in S_n$ chosen at random from $S_n$.
Let $C(\pi)$ denote the $n$ $4$-tuples $(a,b,c,d)$ such that $(a,b,c)$ and $(b,c,d)$ are
consecutive elements of $S(\pi)$. 
Note that, equivalently, $C(\pi)$ can be viewed as the set of consecutive pairs of elements of $S(\pi)$.
A subset $S \subset C(\pi)$ is {\em independent} if any two elements of $S$
are disjoint (so the union of two elements of $S$ consists of $8$ vertices).
For a subset $S \subset C(\pi)$ we say that it is $|S|$-bad
if it is independent and for each $4$-tuple in $S$, the two consecutive edges forming the $4$-tuple appear in the same
block. Our goal is to prove that if $|S|$ is large, the probability of being $|S|$-bad is small.

Consider some $S \subset C(\pi)$ with $s=|S|$ and suppose that $S=\{(x_i,y_i,z_i,w_i) ~|~ i=1,\ldots,s\}$.
Stated otherwise, we pick $s$ positions of $\pi$ at least four apart each and look at the $s$ $4$-tuples where each $4$-tuple
is a sequence of $4$ consecutive elements of $\pi$ starting at one of the chosen positions.
For example, if $n=9$, $\pi=314827956$ and $s=2$ and we are looking at positions $3,7$ then
$S=\{(4,8,2,7),(9,5,6,3)\}$.
What is the probability that $S$ is $s$-bad? (we may assume that the $s$ $4$-tuples in $S$ are
independent otherwise it is not bad by definition; in the last example they are indeed independent).
Consider first the $4$-tuple $(x_1,y_1,z_1,w_1)$. The probability that all the four vertices fall in the same block (meaning that $(x_1,y_1,z_1),\, (y_1,z_1,w_1)$ are two consecutive elements of $S(\pi)$ falling in the same block) is $(q-2)/(n-3)$.
Given that the four elements of $(x_i,y_i,z_i,w_i)$ fall in the same block for $i=1,\ldots,r-1$ (distinct
$4$-tuples may or may not fall in the same block) what is the probability
that also the four vertices of $(x_r,y_r,z_r,w_r)$ fall in the same block?
Suppose we are given the information to which block each of the $4r-1$ vertices of
$x_1,y_1,z_1,w_1,x_2,y_2,z_2,w_2\ldots,x_{r-1},y_{r-1},z_{r-1}w_{r-1},x_r,y_r,z_r$ 
belongs. The probability that $w_r$ also falls in the block to which $x_r,y_r,z_r$ belong is thus at most $(q-2)/(n-4r+1)$.
Hence, the probability that $S$ is $s$-bad is at most
$$
\left(\frac{q-2}{n-4s+1}\right)^s\;.
$$
Let us say that $\pi$ is {\em $s$-bad} if it contains some independent $S \subset C(\pi)$
that is $s=|S|$ bad. Thus, the probability that a randomly chosen $\pi$ is $\lfloor n^{1/3} \rfloor$-bad is at most
$$
\binom{n}{\lfloor n^{1/3} \rfloor} \left(\frac{q-2}{n-4\lfloor n^{1/3} \rfloor+1}\right)^{\lfloor
	n^{1/3} \rfloor}  \ll \frac{1}{n^2}
$$
where we have only used here that $n$ is sufficiently large and that $q = \Theta(n^{1/4})$.
In other words, we have proved that the number of $\lfloor n^{1/3} \rfloor$-bad permutations
is at most $n!/n^2 \le (n-2)!$.

We will say that $\pi$ is {\em very bad} if there is a set of at least $6n^{1/3}$ (not necessarily disjoint) elements of $C(\pi)$ such that for each $4$-tuple in this set, all its four vertices appear in the same block. Observe that since each $4$-tuple in $C(\pi)$ intersects only $6$ other $4$-tuples (recall that the $4$-tuples are consecutive elements of $\pi$), a very bad $\pi$ is also
$\lfloor n^{1/3} \rfloor$-bad, and in particular there are at most $(n-2)!$ very bad permutations.

Let us now bound $\sum_{\pi \notin{\cal L}} g(\pi)$.
By Lemma \ref{l:num-bad}, we know that the number of $\pi \notin{\cal L}$ is at most $n! \cdot \Theta(n^{-1/4})$.
We have just shown that at most $(n-2)!$ of them are very bad (in fact, at most $(n-2)!$ out of all permutations, not just bad ones)
and for them we will use the trivial bound $g(\pi) \le n$.
The others are not very bad, so for them we have $g(\pi) \le 6n^{1/3}$.
Thus,
$$
\sum_{\pi \notin{\cal L}} g(\pi) \le n! \cdot \Theta(n^{-1/4})6n^{1/3}+(n-2)!n\;.
$$
Hence
$$
\frac{\sum_{\pi \notin{\cal L}} g(\pi)}{n!} \le \Theta(n^{1/12})\;.
$$
Now,
\begin{align*}
	\overline{g} & =  \frac{n! \overline{g^*} - \sum_{\pi \notin{\cal L}} g(\pi)}{|{\cal L}|}  \\
	& \ge \overline{g^*} - \frac{\sum_{\pi \notin{\cal L}} g(\pi)}{|{\cal L}|} \\
	& \ge  \overline{g^*} - \frac{2\sum_{\pi \notin{\cal L}} g(\pi)}{n!} \\
	& \ge   \overline{g^*} - \Theta(n^{1/12})\;.
\end{align*}
\end{proof}

We now return to (\ref{e:amgm}) and obtain, using Lemma \ref{l:g^*} and Lemma \ref{l:g}, that
$$
\mathbb{E}[H(G)] \ge \frac{|\L|}{2n}p^{\overline{f}} \ge \frac{|\L|}{2n}p^{n-\overline{g}}
\ge \frac{|\L|}{2n}p^{n-\overline{g^*}+\Theta(n^{1/12})} = \frac{|\L|}{2n}p^{n-\Theta(n^{1/4})}\;.
$$\
Now, since $|\L| \ge n!/2$, since $E(n,p)=p^n(n-1)!/2$ and since
$H_3(n,p,\epsilon) \ge \mathbb{E}[H(G)]$ we obtain from the last inequality that
$$
\frac{H_3(n,p,\epsilon)}{E(n,p)} \ge (1/p)^{\Theta(n^{1/4})}\;.
$$
\end{proof}

\section{Proof of Theorem \ref{t:hypergraphs}}

The major obstacle when trying to generalize the proof of Theorem \ref{t:special} to larger $r$, is that for
$r \ge 4$ we do not have a nice decomposing object such as
a Steiner system with blocks of {\em polynomial} size, as we have in the case $r=3$ where we have used $S(3,q+1,q^4+1)$ where $q^4=\Theta(n)$\footnote{Although Theorem \ref{t:special} was proved for the case
$n=q^4+1$, it is not too difficult to generalize it so that it holds for all $n$, since for every $n$ there is an $n'$ of the form $q^4+1$ where $q$ is a prime power and $n/16 \le n' \le n$, so it is always true that a constant proportion of the $3$-edges of $K_n^3$ can be packed by the Steiner system $S(3,q+1,q^4+1)$; this suffices in order to obtain a result that extends Theorem \ref{t:special} to all $n$.}.
In the case $r=3$, the set of blocks of $S(3,q+1,q^4+1)$  can each be viewed as a set of pairwise edge-disjoint complete $3$-graphs on $q+1$ vertices. This makes the construction and proofs of the various lemmas in the previous section rather smooth. To compensate for the lack of such structure when $r \ge 4$,
we can replace the blocks (i.e. replace pairwise edge-disjoint complete $r$-graphs) with pairwise edge-disjoint {\em sufficiently dense} $r$-graphs of polynomial order that still cover a non-negligible proportion of all $r$-sets of our $n$-set.
This idea is formalized in Definition \ref{d:packing}. We note that the construction we are about to show also works for $r=3$ and, in fact, gives a somewhat stronger result than the one obtained in Theorem \ref{t:special}.

It will be convenient to reformulate Theorem \ref{t:hypergraphs} in the following equivalent way.
\begin{theorem}\label{t:hypergraphs-reform}
	Let $\delta >0$, $\epsilon > 0$, $p \in (0,1)$ and $r \ge 3$ be fixed. Then,
	$$
	\frac{H_r(n,p,\epsilon)}{E(n,p)} \ge 2^{\Omega(n^{\gamma})}
	$$
	where $\gamma=\frac{1}{2}-3\delta$ if $r \ge 4$ and $\gamma=\frac{1}{3}-3\delta$ if $r=3$.
\end{theorem}

Prior to proving Theorem \ref{t:hypergraphs-reform} we need to set up some definitions.
For an $r$-graph $B$ and for a set $X$ of size $1 \le |X| \le r-1$ of vertices of $B$, the {\em degree} of $X$ in $B$ is the number of edges of $B$ that contain $X$. Notice that the degree of $X$ can be at most
$\binom{|V(B)|-|X|}{r-|X|}$.
We now present the definition that replaces the Steiner system used in Theorem \ref{t:special}
to a family of sufficiently dense edge-disjoint $r$-graphs of polynomial size, that we will use in the proof of
Theorem \ref{t:hypergraphs-reform}.
\begin{definition}\label{d:packing}
	For reals $\beta,\delta > 0$ and for positive integers $k \ge 1$, $r \ge 3$ and $n$, 
	an $(n,r,k,\beta, \delta)$-packing of $K_n^r$ is a set $L$ of pairwise edge-disjoint $r$-subgraphs of $K_n^r$ such that the following holds.\\
	(i) Every element of $L$ is an $r$-graph with $q=\lceil n^\beta\rceil$ vertices.\\
	(ii) If $B \in L$ and $X$ is a set of size $1 \le |X| \le r-1$ of vertices of $B$, then the degree of $X$ in $B$ is at least $n^{-\delta}\binom{q-|X|}{r-|X|}$.\\
	(iii) At most a fraction of $\frac{1}{2}$ of the edges of $K_n^r$ are in some element of $L$.\\
	(iv)  $|L| \ge n^{r-r\beta}$ and $|L|$ is a multiple of $k$.\\
	(v) The number of edges of all elements of $L$ is the same.
\end{definition}
\begin{lemma}\label{l:exists}
	Let $0 < \delta < \beta < 1$, $k \ge 1$, $r \ge 3$.
	Then, for all $n$-sufficiently large, an $(n,r,k,\beta,\delta)$-packing of $K_n^r$ exists.
\end{lemma}
\begin{proof}
Throughout the proof we will assume that $n$ is sufficiently large so that the statements involving it hold.
Our construction is probabilistic. Let $q=\lceil n^\beta\rceil$ and let $K = \lceil n^{r-r\beta} \rceil+k'$ so that $0 \le k' < k$ and $K$ is a multiple of $k$. As $q$ will be the size of each element of our constructed packing and $K$ will be the number of elements, Properties (i) and (iv) of Definition \ref{d:packing} hold.
We randomly select (with replacement) $K$ subsets of $[n]$, each of size $q$, denoting the chosen subsets by $V_1,\ldots,V_K$. Let $B_i$ be the subgraph of $K_n^r$ induced by $V_i$, so $B_i$ is isomorphic to $K_q^r$.

Unfortunately, the $B_i$'s are not necessarily edge-disjoint, so, in particular, we cannot claim that $\{B_1,\ldots,B_K\}$ is an $(n,r,k,\beta,\delta)$-packing. We are therefore required to perform several adjustments to our $B_i$'s in order to guarantee the existence of an $(n,r,k,\beta,\delta)$-packing.

For an edge $e \in E(K_n^r)$, let $C(e)=\{i\,|\, e \in E(B_i)\}$ be the set of indices $i$ such that $e$ is an edge of $B_i$ (equivalently $e \subset V_i$).
For a given $i$, the probability that $e \subset V_i$ is $\binom{q}{r}/\binom{n}{r}$. As the $V_i$'s are independently chosen,
we have that $|C(e)|$ has binomial distribution $Bin(K, \binom{q}{r}/\binom{n}{r})$. By our choice of
$K$ and $q$ we have that
$$
\mathbb{E}[|C(e)|]= K\frac{\binom{q}{r}}{\binom{n}{r}} = \Theta(n^{r-r\beta}n^{(\beta-1)r})=\Theta(1)\;.
$$
Hence by a Chernoff bound (see \cite{AS-2004}, A.1.12), we have that
$$
\Pr \left [ |C(e)| > \frac{n^\delta}{4} \right ] < 2^{-\Theta(n^\delta)}.
$$
As there are only $\binom{n}{r} \ll 2^{\Theta(n^\delta)}$ edges in $K_n^r$, we have by the union bound that with high probability,
no edge appears in more than $M=\lfloor \frac{1}{4}n^\delta \rfloor$ distinct $B_i$'s.
So, from here until the end of the proof we assume that this is indeed the case for our $B_i$'s, namely, $|C(e)| \le M$ for all $e \in E(K_n^r)$.

Let $C(e)=\{i_{e,1},\ldots,i_{e,|C(e)|}\}$.
For each $e \in E(K_n^r)$ with $C(e) \neq \emptyset $, we select uniformly at random
an integer $j$ in $\{1,\ldots,M\}$ and proceed as follows. If $j > |C(e)|$ then we remove $e$ from all the
$B_i$'s to which it belongs. If $j \le |C(e)|$ then we remove $e$ from all the
$B_i$'s to which it belongs {\em except} for $B_{i_{e,j}}$. Let $B_i^*$ be the spanning $r$-subgraph of $B_i$ obtained after this edge-removal process.
Observe that now our $B_i^*$ are pairwise edge-disjoint as every edge $e$ is an element of at most one of the $B_i^*$. It is now not difficult (see below) to prove that with high probability each of the $B_i^*$ satisfies the degree condition, namely Property (ii) of Definition \ref{d:packing}. However, we need a bit of additional modifications in order to guarantee that Property (v) of Definition \ref{d:packing} also holds.
To this end, we color each edge of each of the $B_i^*$ either red or white uniformly and independently at random.

We define several random variables.
Let $R(i)$ be the number of red edges of $B_i^*$ and let $W(i)$ be the number of white edges of $B_i^*$.
Observe that each has distribution $Bin(\binom{q}{r}, \frac{1}{2M})$.
Consider some $X \subset V_i$ with $1 \le |X| \le r-1$, and let $R(X)$ be the red degree of $X$ in $B_i^*$.
Let $F(X)$ denote the set of edges of $B_i$ that contain $X$ (namely, the edges before the edge-removal process) and observe that $|F(X)|=\binom{q-|X|}{r-|X|}$.
So, $R(X)$ is the sum of {\em independent} indicator random variables $Y(e)$ for each $e \in F(X)$ where $Y(e)=1$ 
if $e$ is a red edge of $B_i^*$. Observe that $\Pr[Y(e)=1] = 1/(2M)$.
So $R(X)$ has distribution $Bin(\binom{q-|X|}{r-|X|},\frac{1}{2M})$.

All the following three claims are consequences of Chernoff's large deviation inequality for the binomial distribution.

\noindent
{\bf Claim 1.} With probability at least $3/4$, it holds for all $1 \le i \le K$ and for all
	$X \subset V_i$ with $1 \le |X| \le r-1$ that $R(X) \ge n^{-\delta}\binom{q-|X|}{r-|X|}$.\\
{\em Proof.} As $R(X) \sim Bin(\binom{q-|X|}{r-|X|},\frac{1}{2M})$ its expected value
is 
$$
\binom{q-|X|}{r-|X|}\frac{1}{2M} \ge 2\binom{q-|X|}{r-|X|}n^{-\delta} = 
\Omega(n^{\beta(r-|X|)-\delta}) \ge \Omega(n^{\beta-\delta}) \;.
$$
As $\beta > \delta$, the exponent $\beta-\delta$ is positive.
By Chernoff's large deviation approximation (see \cite{AS-2004} Theorem A.1.13), the probability that
$R(X)$ falls below half its expected value is exponentially small in that expected value
hence exponentially small in $n^{\beta-\delta}$. Since the overall number of possible pairs $(i,X)$ is
less than $K q^r$, namely only a polynomial in $n$, we have by the union bound that with probability $1-o_n(1) \ge \frac{3}{4}$, $R(X) \ge n^{-\delta}\binom{q-|X|}{r-|X|}$ for all possible $(i,X)$ pairs.
\qed

\noindent
{\bf Claim 2.} With probability at least $3/4$, it holds for all $1 \le i \le K$ that
$R(i)+W(i) \le \frac{1}{3}\binom{q}{r}$.\\
{\em Proof.}
As $R(i) \sim Bin(\binom{q}{r}, \frac{1}{2M})$ its expected value is
$\Theta(n^{\beta r-\delta}) \ll \frac{1}{12}\binom{q}{r}$.
Since $\beta r -\delta > 0$, we again have by Chernoff's large deviation approximation that the probability that
$R(i)$ is larger than twice its expected value, and in particular larger than $\frac{1}{6}\binom{q}{r}$  is exponentially small in $n^{\beta r-\delta}$. The same holds for $W(i)$. As there are only $K$ pairs
$(W(i),R(i))$ and $K$ is polynomial in $n$, we have by the union bound that with probability $1-o_n(1) \ge \frac{3}{4}$, $R(i)+W(i) \le \frac{1}{6}\binom{q}{r}+\frac{1}{6}\binom{q}{r}=\frac{1}{3}\binom{q}{r}$.
\qed

\noindent
{\bf Claim 3.} With probability at least $3/4$, it holds for all $1 \le i,j \le K$ that
$R(j) \le R(i)+W(i)$.\\
{\em Proof.}
In order for some pair $i,j$ to have $R(j) \ge R(i)+W(i)$ it must hold that either $R(j) \ge 2R(i)$ or
$R(j) \ge 2W(i)$. But each of $R(i),W(i),R(j)$ has distribution $Bin(\binom{q}{r}, \frac{1}{2M})$
whose expectation is $\Theta(n^{\beta r-\delta})$, so in order to have $R(j) \ge 2W(i)$ or $R(j) \ge 2R(i)$
at least one of $R(i),W(i),R(j)$ must deviate by a constant factor from its expected value.
As the probability of this happening is exponentially small in $n^{\beta r-\delta}$ and as there are at most
$K^2$ pairs $(i,j)$ to consider, it follows by the union bound that with probability $1-o_n(1) \ge \frac{3}{4}$,
$R(j) \le R(i)+W(i)$ for all $1 \le i,j \le K$.
\qed

As the statement of each of the last three claims does not hold with probability at most $1/4$ and since there are only three claims, we have that with positive probability all three statements of the last three claims hold.
So let us fix the $B_i^*$ and their red and white edges such that:\\
(a) for all $1 \le i \le K$ and for all
$X \subset V_i$ with $1 \le |X| \le r-1$ it holds that $R(X) \ge n^{-\delta}\binom{q-|X|}{r-|X|}$;\\
(b) for all $1 \le i \le K$ it holds that $R(i)+W(i) \le \frac{1}{3}\binom{q}{r}$;\\
(c) for all $1 \le i,j \le K$ it holds that $R(j) \le R(i)+W(i)$.\\
Let $z = \min_{i=1}^{K} R(i)+W(i)$.
For $1 \le j \le K$ remove from $B_j^*$ an arbitrary set of $R(j)+W(j)-z$ white edges.
This can be done by (c) since if $i$ is such that $R(i)+W(i)=z$ then
$W(j) \ge R(j)+W(j)-z =  R(j)+W(j)-R(i)-W(i)$ follows from (c) by the fact that $R(j) \le R(i)+W(i)$.
Denote by $B_j^{**}$ the spanning $r$-subgraph of $R_j^*$ obtained after this removal of $R(j)+W(j)-z$ white edges.
We claim that $\{B_1^{**},\ldots,B_K^{**}\}$ is an $(n,r,k,\beta,\delta)$-packing of $K_n^r$.
Indeed we have already shown that Properties (i) and (iv) in Definition \ref{d:packing} hold.
Property (v) holds since each $B_j^{**}$ has $z$ edges. Property (ii) follows from (a) and the fact that no
red edges were removed. Property (iii) follows from (b) as the number of edges in all the
$B_i^*$ is at most
$$
K \frac{1}{3}\binom{q}{r} = (1+o_n(1))\frac{1}{3}n^{r-r\beta}\frac{n^{\beta r}}{r!} < \frac{1}{2}\binom{n}{r}\;.
$$
\end{proof}

\noindent
Next we define:
\begin{equation}\label{e:beta}
\beta = \begin{cases}
\frac{1}{2}-\delta & {\rm if~} r \ge 4\\
\frac{1}{3}-\delta & {\rm if~} r = 3\;.
\end{cases}
\end{equation}

\vspace*{2mm}
\noindent
{\bf Proof of Theorem \ref{t:hypergraphs-reform}.}
Let $\delta >0$, $\epsilon > 0$ and $p=\ell/k \in (0,1)$ with $gcd(\ell,k)=1$ be given
and let $r \ge 3$ be given. Let $\beta$ be defined as in (\ref{e:beta}).
Suppose that $n$ is sufficiently large such that
an $(n,r,k,\beta,\delta)$-packing of $K_n^r$, denoted by $L$, exists (this holds by Lemma \ref{l:exists}) and such that $p\binom{n}{r}$ is an integer.
In fact, throughout the proof we will assume without further mention that $n$ is also sufficiently large so that the other stated claims involving it hold.
We need to construct an $(\epsilon,p)$-quasi random $r$-graph with $n$ vertices that contains many Hamiltonian cycles.

Our first task is to partition $L$ into parts of size $k$ each, such that the elements of each part are vertex-disjoint.
\begin{lemma}\label{l:hs}
	There is a partition of $L$ into parts $L_1,\ldots,L_t$ where $t=|L|/k$ and each part is of size $k$.
	Furthermore, any two $r$-graphs in the same part are vertex-disjoint.
\end{lemma}
\begin{proof}
	Define a graph $M$ whose vertex set is $L$ and whose edges connect two elements of $L$ that are not vertex-disjoint. We must prove that $M$ has $|L|/k$ pairwise-disjoint independent sets of size $k$ each
	(or, equivalently, that the complement of $M$ has a factor into cliques of order $k$ each).	
	To this end, we upper-bound the maximum degree of $M$ as follows. Consider some $B \in L$. So by Property (i) of Definition \ref{d:packing}, $B$ is an $r$-graph with $q=\lceil n^\beta\rceil$ vertices. Consider some vertex $v \in V(B)$. There are precisely $\binom{n-1}{r-1}$ edges of $K_n^r$ that contain $v$. But if $v$ is a vertex of some element $B' \in L$, then by
	Property (ii) in Definition \ref{d:packing}, $B'$ contains at least $n^{-\delta}\binom{q-1}{r-1}$ edges that contain $v$. So, the number of elements of $L$ other than $B$ that
	contain $v$ is at most $n^\delta\binom{n-1}{r-1}/\binom{q-1}{r-1}=\Theta(n^{\delta+(1-\beta)(r-1)})$. As $B$ has $q$ vertices, the degree of $B$ in $M$ is at most $\Theta(n^{\delta+\beta+(1-\beta)(r-1)})$.
	On the other hand, the number of vertices of $M$, namely $|L|$, is at least $n^{r-r\beta}$ by Property (iv) of Definition \ref{d:packing}.
	But observe that by the definition of $\beta$, we have that $\beta \le \frac{1}{2}-\delta$ so $n^{\beta+\delta+(1-\beta)(r-1)} = o(n^{r-r\beta})$.
	So $M$ is a graph whose number of vertices is much larger than its maximum degree (their ratio tends to infinity with $n$). As $|L|$ is a multiple of $k$ (Property (iv) in Definition \ref{d:packing}) and $k$ is just a constant independent of $n$, it follows from the Hajnal-Szemer\'edi Theorem \cite{HS-1970} that $M$ has $|L|/k$ pairwise disjoint independent sets of order $k$ each.
\end{proof}

Let $W$ denote the set of edges of $K_n^r$ that are {\em not} in any element of $L$.
\begin{lemma}\label{l:hs-w}
	$|W|$ is a multiple of $k$. Furthermore, $W$ can be partitioned into $t'=|W|/k$ parts, denoted by $W_1,\ldots,W_{t'}$, each of size $k$, and the elements of each $W_i$ are pairwise disjoint.
\end{lemma}
\begin{proof}
Recall that $p \binom{n}{r}$ is an integer , but since $p=\ell/k$ and $gcd(\ell,k)=1$ this implies that
$\binom{n}{r}$ is a multiple of $k$. On the other hand $|W|=\binom{n}{r}-\sum_{B \in L} |E(B)|$.
But by Property (v) in Definition \ref{d:packing}, each term in the sum has the same value and the number of terms is $|L|$ which is a multiple of $k$.
Hence, $|W|$ is a multiple of $k$ as well.
As in the proof of Lemma \ref{l:hs}, we can use the Hajnal-Szemer\'edi Theorem  to obtain the desired partition.
Define a graph $M$ whose vertex set is $W$ and whose edges connect two elements of $W$ that are not disjoint.
Consider some $e \in W$. As $e$ is an $r$-set, consider some vertex $v \in e$.
As there are precisely $\binom{n-1}{r-1}$ edges of $K_n^r$ that contain $v$,
the degree of $e$ in $M$ is $O(rn^{r-1})=O(n^{r-1})$. On the other hand,
the number of vertices of $M$, namely $|W|$, is at least $\Theta(n^r)$ as Property (iii) of Definition
\ref{d:packing} asserts that $|W| \ge \frac{1}{2}\binom{n}{r}$.
So $M$ is a graph whose number of vertices is much larger than its maximum degree (their ratio tends to infinity with $n$). As $|W|$ is a multiple of $k$ and $k$ is just a constant independent of $n$, it follows from the Hajnal-Szemer\'edi Theorem that $M$ has $|W|/k$ pairwise disjoint independent sets of order $k$ each.
\end{proof}

Based on the partitions of $L$ and $W$ obtained in Lemmas \ref{l:hs} and \ref{l:hs-w}, we can now construct our $(\epsilon,p)$-quasi random $r$-graph $G$ on vertex set $[n]$.
From each $L_i$ where $i=1,\ldots,t$ choose precisely $\ell$ of its elements at random, all $t$ choices made independently and overall there are $\ell t$ chosen elements (recall that the elements are $r$-subgraphs of $K_n^r$). Add to $G$ all edges of the $\ell t$ chosen $r$-subgraphs.
The edges belonging to the remaining $k-\ell$ elements of each $L_i$ are {\em non-edges} of $G$.
This defines for all edges of $K_n^r$ other than $W$ whether they are edges or non-edges of $G$.
From each $W_j$ where $j=1,\ldots,t'$, choose precisely $\ell$ of its elements at random, all $t'$ choices made independently and overall there are $\ell t'$ chosen elements (recall that these elements are edges of $K_n^r$).   Add to $G$ all the chosen $\ell t'$ edges. The remaining $k-\ell$ elements of each $W_j$ are non-edges of $G$.
Notice that the obtained $r$-graph $G$ has density precisely $p=\ell/k$, as required.
Furthermore since each $r$-set is an edge with probability $p$ and the event that a given $r$-set
is an edge depends only on $O(k q^r)=o(n^r)$ other $r$-sets, it is straightforward that for $\epsilon > 0$, if $q$ (and hence $n$) is sufficiently large, then $G$ is $(\epsilon,p)$-quasi random with probability $1-o_n(1)$.
We will prove that with high probability, $G$ as above has many Hamiltonian cycles, enough to give Theorem \ref{t:hypergraphs-reform}.

Consider a permutation $\pi$ of $[n]$ together with its $r$-set $S(\pi)$ (recall that the
$r$-set of a permutation consists of the $n$ edges of the Hamiltonian cycle in $K_n^r$ corresponding to the permutation).
We call $\pi$ {\em good} if no $L_i$ contains two edges of $S(\pi)$ in distinct elements of $L_i$
and no $W_j$ contains two elements of $S(\pi)$. Otherwise, it is {\em bad}.
Note that in a good permutation, $L_i$ is allowed to contain several edges of $S(\pi)$ as long as they are all in the same element of $L_i$.
For a good permutation $\pi$ let
$$
f(\pi) = |\{i\,:\, \exists B \in L_i, \exists e \in S(\pi),  e \in E(B)\}|+
|\{j\,:\, \exists e \in W_j \cap S(\pi)\}|\;.
$$
So $f(\pi)$ is the number of parts $L_i$ that contain an element that contains an edge of $S(\pi)$
plus the number of parts $W_j$ that contain an edge of $S(\pi)$.
Observe that trivially $f(\pi) \le |S(\pi)|=n$. The following lemma is analogous to Lemma \ref{l:prob}
since the probability of an element of $L_i$ to be chosen for taking all its edges to $G$ is $p=\ell/k$
and since the probability of an edge of $W_j$ to be chosen as an edge in $G$ is $\ell/k$ as well.
\begin{lemma}\label{l:prob-general}
	Let $\pi$ be a good permutation. The probability that $S(\pi)$ is a Hamiltonian cycle of $G$
	is $(\ell/k)^{f(\pi)}=p^{f(\pi)}$. \qed
\end{lemma}
The following lemma is analogous to Lemma \ref{l:num-bad} but is somewhat more technical as we need to account for the $W_j$'s and as the elements of $L_i$ are not necessarily complete $r$-graphs.
\begin{lemma}\label{l:num-bad-general}
	The number of bad permutations is at most $n! \cdot \Theta(n^{-\delta r})$.
\end{lemma}
\begin{proof}
	It is equivalent, although more convenient, to upper-bound the probability that
	a random permutation $\pi \in S_n$ is bad. So, let $\pi \in S_n$ be drawn at random and recall that its
	$r$-set $S(\pi)$ contains $n$ elements.
	Let $e,f \in S(\pi)$ be two distinct edges. First observe that if $e \cap f \neq \emptyset$ then
	$e$ and $f$ never appear in two different elements of the same $L_i$ since the $r$-subgraphs in $L_i$ are
	pairwise vertex-disjoint. Similarly, $e$ and $f$ never appear together in the same $W_j$ since all elements of $W_j$ are pairwise disjoint. So we can assume that $e \cap f = \emptyset$.
	
	Suppose that $e$ is an edge of some $B \in L_i$. What is the probability that $f$ is an edge of one of the other
	$k-1$ elements of $L_i$? Let $B'$ be such an element and recall that $B'$ and $B$ are vertex-disjoint. As $f$ is a completely random $r$-set of
	$[n] \setminus e$ (recall that $\pi$ is a random permutation), for the event $f \in E(B')$ to occur it must be that all $r$ vertices of $f$ are
	in the set of $q$ vertices of $B'$. The latter occurs with probability $O((q/n)^r)=O(n^{r(\beta-1)})$.
	As there are $k-1$ choices for $B'$, we obtain by the union bound that the probability that $e$ and $f$ are in two distinct elements of the same part $L_i$ is at most $O(kn^{r(\beta-1)})=O(n^{r(\beta-1)})$.
	
	Suppose next that $e$ is in some $W_j$. What is the probability that $f$ is one of the other $k-1$ elements of
	$W_j$? As $f$ is a completely random $r$-set of $[n] \setminus e$, the event $f \in W_j$ occurs with probability $\Theta(k n^{-r}) = \Theta(n^{-r}) \le O(n^{r(\beta-1)})$.

	As there are less than $n^2$ pairs of distinct elements of $S(\pi)$ it follows from the union bound that the probability that $\pi$ is bad is at most $\Theta(n^2n^{r(\beta-1)})=\Theta(n^{2-r+r\beta})$.
	But recall (\ref{e:beta}) that $\beta=\frac{1}{2}-\delta$ if $r \ge 4$ and $\beta=\frac{1}{3}-\delta$ if $r=3$
	hence the latter probability is at most $\Theta(n^{-\delta r})$ and the result follows.
\end{proof}

Let $\L$ denote the set of good permutations. Then by the last lemma, $|\L| \ge n! \cdot (1-\Theta(n^{-\delta r}))$. Observe also that by Lemma \ref{l:prob-general}, the random variable $H(G)$ which is the number of Hamiltonian cycles in $G$ satisfies
\begin{equation}\label{e:EHG-general}
\mathbb{E}[H(G)] \ge \frac{1}{2n}\sum_{\pi \in \L} p^{f(\pi)}\;.
\end{equation}
Although not crucial, a point to observe is that unlike (\ref{e:EHG}) where we have an equality, in (\ref{e:EHG-general}) we only have an inequality because a bad permutation has a small chance of inducing a Hamiltonian cycle, as here we choose $\ell$ elements of $L_i$ for edge-containment in $G$, as opposed to only one element of $L_i$ in the proof of Theorem \ref{t:special}.

Let $\overline{f}$ denote the average of $f(\pi)$ ranging over all $\pi \in \L$. Then we have, analogous to
Lemma \ref{l:amgm} using the inequality of arithmetic and geometric means:
\begin{lemma}\label{l:amgm-general}
	\begin{equation}\label{e:amgm-general}
	\mathbb{E}[H(G)] \ge \frac{|\L|}{2n}p^{\overline{f}}\;.
	\end{equation} \qed
\end{lemma}

As in the previous section, in order to lower-bound (\ref{e:amgm-general}) we need to use the notions of
$g(\pi)$, $\overline{g}$ and $\overline{g^*}$. Their definitions remain the same:
$g(\pi)$ is the number of {\em consecutive} pairs of elements of $S(\pi)$ that fall into the same element of $L$ (however, recall that now some elements of $S(\pi)$ may not even be in any element of $L$ as they may be in $W$),
$\overline{g}$ is the average of $g(\pi)$ ranging over all $\pi \in \L$ (good permutations) and $\overline{g^*}$ is the average of $g(\pi)$ over all $n!$ permutations. Recall also that $\overline{f} \le n - \overline{g}$.
Analogous to Lemma \ref{l:g^*} we have the following (somewhat more involved) lemma:
\begin{lemma}\label{l:g^*-general}
	$$
	\overline{g^*} = \Theta(n^{\beta-2\delta})\;.
	$$
\end{lemma}
\begin{proof}
	Again, it is convenient to consider a permutation $\pi$ chosen at random from $S_n$.
	Consider two consecutive elements of $S(\pi)$. Denote them by $(a_1,\ldots,a_r)$ and $(a_2,\ldots,a_{r+1})$.
	We compute the probability $p^*$ that they are both in the same element of $L$.
	Clearly
	$$
	p^* = \Pr[(a_2,\ldots,a_{r+1}) \in E(B) \;|\; (a_1,\ldots,a_r) \in E(B)] \cdot 
	\Pr[(a_1,\ldots,a_r) \in E(B) {\rm ~where~} B \in L] \;.
	$$
	By Property (ii) of Definition \ref{d:packing}, the minimum degree of each element of $L$ is at
	least $n^{-\delta}\binom{q-1}{r-1}=\Theta(n^{\beta(r-1)-\delta})$. Hence, the number of edges of each element of $L$ is at least $\Theta(n^{\beta r-\delta})$. As $|L| \ge n^{r-r\beta}$, the number of edges of $K_n^r$ appearing in some element of $L$ is at least $\Theta(n^{r-\delta})$.
	As $K_n^r$ has $\Theta(n^r)$ edges, this implies that
	$\Pr[(a_1,\ldots,a_r) \in E(B) {\rm ~where~} B \in L] \ge \Theta(n^{-\delta})$.
	So it remains to estimate $\Pr[(a_2,\ldots,a_{r+1}) \in E(B) \;|\; (a_1,\ldots,a_r) \in E(B)]$.
	Suppose that we are given that $(a_1,\ldots,a_r) \in E(B)$.
	In particular $a_1,\ldots,a_r$ are all vertices of $B$. Now, consider the $(r-1)$-set $\{a_2,\ldots,a_r\}$.
	By Property (ii) of Definition \ref{d:packing}, there are at least $n^{-\delta}(q-r+1)=\Theta(n^{\beta-\delta})$
	edges of $B$ that contain $\{a_2,\ldots,a_r\}$. Hence, $\Pr[(a_2,\ldots,a_{r+1}) \in E(B) \;|\; (a_1,\ldots,a_r) \in E(B)]$ is at least the probability that $a_{r+1}$ is a vertex of one of these $\Theta(n^{\beta-\delta})$ edges.
	This occurs with probability $\Theta(n^{\beta-\delta})/(n-r) = \Theta(n^{\beta-1-\delta})$.
	Hence we have proved that $p^* = \Theta(n^{\beta-1-2\delta})$.
	
	As there are $n$ distinct pairs of consecutive elements of $S(\pi)$, the expected value of
	$g(\pi)$ for a randomly chosen permutation is $n \cdot \Theta(n^{\beta-1-2\delta})=\Theta(n^{\beta-2\delta})$, as claimed.
\end{proof}

\begin{lemma}\label{l:g-general}
	$\overline{g} \ge \overline{g^*} - \Theta(n^{\beta+\delta/2-\delta r}) $.
\end{lemma}
\begin{proof}
Consider some $\pi \in S_n$ chosen at random from $S_n$. Let
$$
C(\pi)=\{e \cup f \;|\; e,f \in S(\pi) \;,\; |e \cup f|=r+1\}\;.
$$
Namely, $C(\pi)$ is the set of all consecutive $(r+1)$-sets of $\pi$ and can alternatively be viewed as the set of pairs of consecutive elements of $S(\pi)$. In particular, $|C(\pi)|=n$.
A subset $S \subset C(\pi)$ is {\em independent} if any two elements of $S$
are disjoint (so the union of two elements of $S$ consists of $2(r+1)$ vertices).
For a subset $S \subset C(\pi)$ we say that it is $|S|$-bad
if it is independent and for each $(r+1)$-set in $S$, the two consecutive edges $e,f \in S(\pi)$ forming it appear in the same element of $L$.
Our goal is to prove that if $|S|$ is large, the probability of being $|S|$-bad is small.

Consider some $S \subset C(\pi)$ with $s=|S|$ and suppose that
$S=\{(a_{i,1},a_{i,2},\ldots,a_{i,r+1}) ~|~ i=1,\ldots,s\}$.
What is the probability that $S$ is $s$-bad? (we may assume that $S$ is
independent otherwise it is not bad by definition).
Consider first the $(r+1)$-tuple $(a_{1,1},a_{1,2},\ldots,a_{1,r+1})$.
The probability that both of the edges $(a_{1,1},a_{1,2},\ldots,a_{1,r})$ and $(a_{1,2},a_{1,3},\ldots,a_{1,r+1})$ appear in the same element of $L$ is at most the probability that
all the $r+1$ vertices of the $(r+1)$-tuple are vertices of the same $B \in L$.
The latter is at most the probability that $a_{1,r+1}$ appears in some $B \in L$ given that
$\{a_{1,1},a_{1,2},\ldots,a_{1,r}\}$ are vertices of $B$, hence the probability is at most $(q-r)/(n-r)$.
Given that all the $r+1$ vertices of the $(r+1)$-tuple $(a_{i,1},a_{i,2},\ldots,a_{i,r+1})$ fall in the same element of $L$, for $i=1,\ldots,z-1$ (distinct $(r+1)$-tuples may or may not fall in the same element of $L$), what is the probability
that also all vertices of $(a_{z,1},a_{z,2},\ldots,a_{z,r+1})$ fall in the same element of $L$?
Suppose we are given the information to which element of $L$ each of the $(r+1)z-1$ vertices
$a_{1,1},a_{1,2},\ldots,a_{z-1,r+1},a_{z,1},\ldots,a_{z,r}$ 
belongs. The probability that $a_{z,r+1}$ also belongs to the element to which $a_{z,1},\ldots,a_{z,r}$ belong is thus at most $(q-r)/(n-(r+1)z+1)$.
Hence, the probability that $S$ is $s$-bad is at most
$$
\left(\frac{q-r}{n-(r+1)s+1}\right)^s\;.
$$
Let us say that $\pi$ is {\em $s$-bad} if it contains some independent $S \subset C(\pi)$
that is $s=|S|$ bad. Now, let $\alpha = \beta+\delta/2$ (so by the definition of $\beta$ we have that
$\alpha=\frac{1}{2}-\delta/2$ if $r \ge 4$ and $\alpha=\frac{1}{3}-\delta/2$ if $r=3$).
Thus, the probability that a randomly chosen $\pi$ is $\lfloor n^\alpha \rfloor$-bad is at most
$$
\binom{n}{\lfloor n^\alpha \rfloor} \left(\frac{q-r}{n-(r+1)\lfloor n^\alpha \rfloor+1}\right)^{\lfloor
	n^\alpha \rfloor}  \ll \frac{1}{n^2}
$$
where we have only used here that $n$ is sufficiently large, that $q = \Theta(n^\beta)$,
and that $\alpha > \beta$.
In other words, we have proved that the number of $\lfloor n^\alpha \rfloor$-bad permutations
is at most $n!/n^2 \le (n-2)!$.

We will say that $\pi$ is {\em very bad} if there is a set of at least $2rn^{\alpha}$ (not necessarily disjoint) elements of $C(\pi)$ such that for each $(r+1)$-tuple in this set, all its vertices appear in the same element of $L$. Observe that since each $(r+1)$-tuple in $C(\pi)$ intersects only $2r$ other $(r+1)$-tuples (recall that the $(r+1)$-tuples are consecutive elements of $\pi$), a very bad $\pi$ is also
$\lfloor n^\alpha \rfloor$-bad, and in particular there are at most $(n-2)!$ very bad permutations.

Let us now bound $\sum_{\pi \notin{\cal L}} g(\pi)$.
By Lemma \ref{l:num-bad-general}, we know that the number of $\pi \notin{\cal L}$ is at most
$n! \cdot \Theta(n^{-\delta r})$.
We have just shown that at most $(n-2)!$ of them are very bad (in fact, at most $(n-2)!$ out of all permutations, not just bad ones)
and for them we will use the trivial bound $g(\pi) \le n$.
The others are not very bad, so for them we have $g(\pi) \le 2rn^{\alpha}$.
Thus,
$$
\sum_{\pi \notin{\cal L}} g(\pi) \le n! \cdot \Theta(n^{-\delta r})2r n^{\alpha}+(n-2)!n\;.
$$
Hence
$$
\frac{\sum_{\pi \notin{\cal L}} g(\pi)}{n!} \le \Theta(n^{\alpha-\delta r})\;.
$$
Now,
\begin{align*}
\overline{g} & =  \frac{n! \overline{g^*} - \sum_{\pi \notin{\cal L}} g(\pi)}{|{\cal L}|}  \\
& \ge \overline{g^*} - \frac{\sum_{\pi \notin{\cal L}} g(\pi)}{|{\cal L}|} \\
& \ge  \overline{g^*} - \frac{2\sum_{\pi \notin{\cal L}} g(\pi)}{n!} \\
& \ge   \overline{g^*} - \Theta(n^{\alpha-\delta r})\;.
\end{align*}
\end{proof}

Using (\ref{e:amgm-general}), Lemma \ref{l:g^*-general} and Lemma \ref{l:g-general},
$$
\mathbb{E}[H(G)] \ge \frac{|\L|}{2n}p^{\overline{f}} \ge \frac{|\L|}{2n}p^{n-\overline{g}}
\ge \frac{|\L|}{2n}p^{n-\overline{g^*}+\Theta(n^{\beta+\delta/2-\delta r})} = \frac{|\L|}{2n}p^{n-\Theta(n^{\beta-2\delta})}\;.
$$
Since $|\L| \ge n!/2$, since $E(n,p)=p^n(n-1)!/2$ and since
$H_r(n,p,\epsilon) \ge \mathbb{E}[H(G)]$ we obtain from the last inequality that
$$
\frac{H_r(n,p,\epsilon)}{E(n,p)} \ge (1/p)^{\Theta(n^{\beta-2\delta})} \ge 2^{\Omega(n^{\gamma})}
$$
where $\gamma=\beta-2\delta$. So, indeed $\gamma=\frac{1}{2}-3\delta$ if $r \ge 4$ and $\gamma=\frac{1}{3}-3\delta$ if $r=3$. This proves Theorem \ref{t:hypergraphs-reform}. \qed
	
\section{\texorpdfstring{$r$}{r}-graphs and \texorpdfstring{$H_r(n,p)$}{Hr(n,p)}}

In this section we prove Theorem \ref{t:hypergraphs-2}.
We present an explicit construction suitable for certain densities and then use Lemma \ref{l:q-p} to cover all densities.

Let $k \ge r \ge 3$ be fixed.
The balanced complete $k$-partite $r$-graph $T_r(n,k)$ is defined as follows. It has $n$ vertices
partitioned into $k$ parts, each of size either
$\lceil n/k \rceil$ or $\lfloor n/k \rfloor$ and its edge set consists of all partial transversals of order $r$, namely, all
$r$-subsets of vertices that intersect each part at most once.

If we assume that $n$ is a multiple of $k$ then the number of edges of $T_r(n,k)$ is $\binom{k}{r}(n/k)^r$. Without this assumption, its density is
\begin{equation}\label{e:d-nkr}
\frac{k!}{(k-r)!k^r}(1 \pm \Theta(\textstyle \frac{1}{n}))\;.
\end{equation}
Notice that for every fixed $p$ and $r$, for a sufficiently large $k$ it holds that the density is larger than $p$. In what follows we will assume that $k > r$ (if $k=r$ and $n$ is not a multiple of $r$ then
$T_r(n,k)$ is not Hamiltonian and we would like to avoid dealing with this special case).
\begin{lemma}\label{l:nkr}
$$
H(T_r(n,k)) \ge \frac{1}{2}\left(\frac{k-r+1}{k}\right)^n (n-1)! (k-r+1)^{-\Theta(\sqrt{n})}\;.
$$
\end{lemma}
\begin{proof}
A permutation $\pi$ of the $n$ vertices of $T_r(n,k)$ is called {\em good} if
every $r$ consecutive vertices are from distinct parts and this holds cyclically.
Hence, for example, if $n=8$, $r=3$ and $k=4$
then a permutation of the form $ABCADBCD$ is good while $ABCDBCAD$ is not (here letters represent vertices from a part named by that letter).
By the definition of $T_r(n,k)$, every good permutation corresponds to a Hamiltonian cycle
and recall each Hamiltonian cycle gives rise to $2n$ good permutations. Thus,
$H(T_r(n,k)) \ge P(n,k,r)/(2n)$ where $P(n,k,r)$ is the number of good permutations.
It is difficult to obtain an exact closed form for $P(n,k,r)$ but it is possible to obtain
a (rather tight) lower bound for it.

A word of length $t$ over $[k]$ is called {\em admissible} if every $r$ consecutive letters are distinct
(here we do not require that this holds cyclically). Thus, for example if $r=3$, $k=4$ and $t=6$, the word $341324$ is admissible. The overall number of admissible words of length $t$ over $[k]$ is 
$k(k-1)\cdots(k-r+2)(k-r+1)^{t-r+1}$. In our setting we would like to count admissible words of length $t$
where $t$ is only slightly less than $n$ and we would also like each letter of $[k]$ to appear roughly its expected number of $t/k$ times. We can quantify this requirement as follows.

Consider the symmetric probability space of all admissible words of length $t$ over $[k]$.
For $\ell \in [k]$, let $X_\ell$ denote the random variable counting the number of occurrences of $\ell$
in  the chosen word. By symmetry, $E[X_\ell]=t/k$.
For $i=0,\ldots,t$, let $X_{\ell,i}$ denote the random variable which equals the {\em expected} number of occurrences of $\ell$ {\em given} the first $i$ letters of the chosen word.
Trivially, $X_{\ell,0}=E[X_\ell]=t/k$ and $X_{\ell,t}=X_\ell$ since at stage $t$ the whole word is known.
Notice that the sequence $X_{\ell,0},\ldots,X_{\ell,t}$ is a Doob martingale (see \cite{AS-2004}) and furthermore $|X_{\ell,i+1}-X_{\ell,i}| \le 1$ for $i=0,\ldots,t-1$ since at stage $i+1$ a single letter (which may or may not be equal to $\ell$) is exposed. Hence, by Azuma's inequality,
$$
\Pr[ |X_\ell - t/k| > k\sqrt{t} ] < 2e^{-k^2/2} < \frac{1}{2k}\;.
$$
Taking the union bound for all letters $\ell \in [k]$ we obtain that at least half of the admissible
words of length $t$ have the property that each letter appears at least $t/k-k\sqrt{t}$ times and
at most $t/k+k\sqrt{t}$ times. Call such admissible words {\em feasible}. Thus, there are at least
$\frac{1}{2}k(k-1)\cdots(k-r+2)(k-r+1)^{t-r+1}$ feasible words of length $t$.

Recall that $T_r(n,k)$ has $k$ parts, so denote them by $A_1,\ldots,A_k$
where $|A_i|$ is either $\lfloor n/k \rfloor$ or $\lceil n/k \rceil$.
A word of length $n$ over $[k]$ is called {\em good} if (i) every $r$ consecutive letters are
distinct and this holds cyclically and (ii) for each $\ell \in [k]$, letter $\ell$ occurs
precisely $|A_\ell|$ times. Observe that each good word corresponds to precisely
$\Pi_{\ell=1}^k |A_\ell|!$ good permutations.

Let $t=n-2k(k-1)^2\lceil \sqrt{n} \rceil -(k-1)^2-kr$ (so $t$ is very long).
We will prove that each feasible word of length $t$ is a {\em prefix} of some good word
of length $n$. As we have a lower bound for the number of feasible words, this will lower-bound the number of
good words which, in turn, lower bounds $P(n,k,r)$.

So, consider some feasible word $w$ of length $t$. Let $d$ be the difference between the letter that appears
least in $w$ and the letter than appears most in $w$. By the definition of feasible words,
$d \le 2k\sqrt{t} \le 2k\sqrt{n}$. Repeat the following process. If $d=0$ the process is done.
Otherwise, let $\ell$ be a most common letter ($\ell$ may not be unique, there could be up to $k-1$ candidates for $\ell$). Append to $w$ a permutation of the $k-1$ letters $[k] \setminus \ell$ such that the
property that every $r$ consecutive letters are distinct remains. Note that this is possible since $k > r$.
So, for example if $k=4$, $r=3$, and $3$ is most common in $w$ then if $w$ ends with, say $432$ then we can append $w$ with $412$ (or with $142$). After this step, $d$ did not increase. It may have decreased by $1$
or otherwise the number of letters that are most common decreased by $1$. So, after repeating this step at most $k-1$ times, $d$ must decrease. So after repeating these steps at most $(k-1)2k\sqrt{n}$ times we obtain
a word $w'$ with $d=0$ and its length is at most $t+(k-1)^2 2k\sqrt{n}$. Observe that by the choice of
$t$ we have that $w'$ is of length at most $n-(k-1)^2-kr$ and each letter occurs precisely $|w'|/k$ times
in $w'$.

Next we take care of slight imbalances due to the fact that $n$ is not necessarily a multiple of $k$.
Assume therefore that $n$ is not a multiple of $k$ and let $1 \le q \le k-1$ be the number of parts with
$\lfloor n/k \rfloor$ vertices. Without loss of generality, assume these are $A_1,\ldots,A_q$.
Repeat the following for all $\ell=1,\ldots,q$. Append to $w'$ a permutation of the $k-1$ letters $[k] \setminus \ell$ such that the property that every $r$ consecutive letters are distinct remains. 
Again note that this is possible since $k > r$. After this process ends we obtain a word $w^*$
of length at most $n-(k-1)^2-kr+(k-1)q \le n-kr$ such that either $d=0$ (if $n$ is a multiple of $k$) or else $d=1$ and the $q$ letters in $[q]$ have one less occurrence than the $q$ letters in
$[k] \setminus [q]$. In particular, $n-|w^*| \ge kr$ and is a multiple of $k$.
Arbitrarily append permutations of $[k]$ to $w^*$ until its length is precisely $n-kr$
while keeping the property that every $r$ consecutive letters are distinct, so we can
assume that $|w^*|=n-kr$. Finally we now have to complete $w^*$ to a good word by adding an amount
of $r$ 
permutations of $[k]$ such that the property of every $r$ consecutive letters holds cyclically.
Assume the first $r$ letters of $w^*$ are $\ell_1,\ldots,\ell_r$ in this order and recall that they are distinct.
Append to $w^*$  a permutation of $[k]$ so that $\ell_1$ is last in the added permutation and which keeps the property that every $r$ consecutive letters are distinct. Next append a permutation of $[k]$ so that the
suffix of the added permutation is $\ell_1\ell_2$ while keeping the property. Continue this process so that at stage $i$ append a permutation of $[k]$ so that its suffix is $\ell_1\cdots\ell_{i-1}\ell_i$. After $r$ stages, the obtained word is of length $n$
and its suffix (and its prefix) is $\ell_1\cdots\ell_{r-1}\ell_r$ so the requirement holds cyclically and so the word is good.

As each feasible word has been extended to a good word, the number of good words is at least $\frac{1}{2}k(k-1)\cdots(k-r+2)(k-r+1)^{t-r+1}$.
Thus,
\begin{align*}
H(T_r(n,k)) & \ge \frac{1}{2n} P(n,k,r) \\
& \ge \frac{1}{4n}k(k-1)\cdots(k-r+2)(k-r+1)^{t-r+1} \cdot \Pi_{\ell=1}^k |A_\ell|! \\
& \ge \frac{1}{4n}(k-r+1)^t \Pi_{\ell=1}^k |A_\ell|! \\
& \ge \frac{1}{4n}(k-r+1)^n (\lfloor n/k \rfloor !)^k (k-r+1)^{-2k(k-1)^2\lceil \sqrt{n} \rceil -(k-1)^2-kr}\\
& = \frac{1}{2}\left(\frac{k-r+1}{k}\right)^n (n-1)! (k-r+1)^{-\Theta(\sqrt{n})}\;.
\end{align*}
\end{proof}

We can now complete the proof of Theorem \ref{t:hypergraphs-2} using Lemma \ref{l:nkr} and Lemma \ref{l:q-p}.
Let $p \in (0,1) $ and let $k \ge r+1$ be the smallest integer such that $\frac{k!}{(k-r)!k^r} > p$.
By (\ref{e:d-nkr}), if $n$ is sufficiently large, then the density of $T_r(n,k)$, denoted by $q$, is larger
than $p$. We apply Lemma \ref{l:q-p} with $G=T_r(n,k)$. By that lemma and by Lemma \ref{l:nkr}  we obtain
\begin{align*}
H_r(n,p) & \ge (p/q)^n e^{-2/p} H(T_r(n,k))(1-o(1))\\
& \ge (p/q)^n \frac{1}{2}\left(\frac{k-r+1}{k}\right)^n (n-1)! (k-r+1)^{-\Theta(\sqrt{n})}\\
& = E(n,p) \left(\frac{k-r+1}{qk}\right)^n (k-r+1)^{-\Theta(\sqrt{n})}\\
& = E(n,p) \left(\frac{(k-r+1)!k^{r-1}}{k!}\right)^n (k-r+1)^{-\Theta(\sqrt{n})}\\
& = E(n,p) \left(\frac{(k-r+1)!k^{r-1}}{k!}\right)^{n-o(n)}\;.
\end{align*}
Finally, observe that since $r > 2$, it holds that $\frac{(k-r+1)!k^{r-1}}{k!} > 1$, hence the theorem.
\qed

\section*{Acknowledgment}
I thank the referees for very useful comments.

\end{document}